\documentclass[12pt,a4paper]{article}      
\usepackage{hyperref}
\usepackage{graphicx}
\usepackage{enumerate}
\usepackage{amssymb, amsmath, amsthm}

\usepackage{tikz}
\usetikzlibrary{arrows,automata}

\newtheorem{proposition}{Proposition}[section] 

\newtheorem{thm}[proposition]{Theorem}
\newtheorem{lemma}[proposition]{Lemma}

\newtheorem{remark}[proposition]{Remark}
\newtheorem{example}[proposition]{Example}

\usepackage[T1]{fontenc}
\usepackage{lmodern}

\newcommand{\N}{\ensuremath{{\mathbb N}}}

\newcommand{\R}{\ensuremath{{\mathbb R}}}

\newcommand{\D}{\ensuremath{{\mathcal D}}}

\newcommand{\E}{\ensuremath{{\mathbb E}}}
\newcommand{\Pro}{\ensuremath{{\mathbb P}}}

\begin{document}
\title{Optimal stopping of strong Markov processes}
\author{S\"oren Christensen\thanks{Christian-Albrechts-Universit\"at,
    Mathematisches Seminar, Ludewig-Meyn-Str.4, D-24098 Kiel, Germany,
    email: {christensen}@math.uni-kiel.de.} , Paavo
  Salminen\thanks{\AA bo Akademi, Department of Natural Sciences,
    Mathematics and Statistics, \small FIN-20500 \AA bo, Finland, email: phsalmin@abo.fi, tbao@abo.fi} , and Bao Quoc Ta \footnotemark[2]}
    \date{}

\maketitle

\begin{abstract}
We characterize the value function and the optimal stopping time for a
large class of optimal stopping problems where the underlying process
to be stopped is a fairly general Markov process. The main result 
is inspired by recent findings for L\'evy processes obtained essentially via
the Wiener-Hopf factorization. The main ingredient in our approach is
the representation of the $\beta$-excessive functions as expected
suprema. A variety of examples is given.
\end{abstract}
%


\section{Introduction}
\label{intro}

Consider a real-valued strong Markov process $X=(X_t)_{t\geq
  0}$. Let $\mathbb{P}_x$ and $\mathbb{E}_x$ stand for the probability 
and the expectation, respectively, associated with $X$ when initiated
from $x.$ The natural filtration generated by $X$ is denoted by $\mathcal F=(\mathcal{F}_t)_ {t\geq0}$  
and $\mathcal{M}$ is the set of all stopping times with
respect to $\mathcal F.$

We are interested in studying the following optimal stopping problem:
Find a function $V$ (value function) and a stopping time $\tau^*$
(optimal stopping time) such that 
\begin{equation}\label{OSP}\tag{*}
V(x):=\sup_{\tau\in\mathcal{M}}\mathbb{E}_x\Big({\rm e}^{-\beta\tau}G(X_\tau)\Big)
=\mathbb{E}_x\Big({\rm e}^{-\beta\tau^*}G(X_{\tau^*})\Big),
\end{equation}
where the function $G$ (reward function) is assumed to have some
regularity properties to be specified later. 
On $\{\tau=\infty\}$ we define 
\[{\rm e}^{-\beta\tau}G(X_\tau)=\limsup_{t\rightarrow\infty}{\rm e}^{-\beta t}G(X_t).\]
Recall that a
non-negative, measurable function $u$ is called $\beta$-excessive for
$X$ if the following two conditions hold:
\[\mathbb{E}_x\Big({\rm e}^{-\beta t}u(X_t)\Big)\leq u(x)\ \ \forall\ t\geq
  0,\, x\in S,\]
\[\lim_{t\rightarrow 0} \mathbb{E}_x\Big({\rm e}^{-\beta
  t}u(X_t)\Big)=u(x)
\ \ \forall\ x\in S.\]
A basic result, see Theorem 1 p. 124 in Shiryayev \cite{shiryayev78},
is that if the reward function $G$ is
lower semicontinuous then the value function $V$ exists and is
characterized as the smallest $\beta$-excessive majorant of $G.$

During the last decade the theory
of optimal stopping for L\'evy processes has been developed strongly. This research has been
focused 
on analyzing the validity of the smooth pasting condition, see
Alili and Kyprianou \cite{AK}, Christensen and Irle \cite{CI}, on the one hand, and 
solving particular problems, see Boyarchenko and
Levendorskii \cite{BL}, 
Mordecki \cite{M}, Novikov and
Shiryayev \cite{NS}, \cite{NS2}, Kyprianou and Surya \cite{KS}, Surya
\cite{S}, Mordecki and Salminen \cite{MS}, Deligiannidis, Le and Utev
\cite{DLU}, on the other hand. The methodology to derive explicit
solutions is in many cases 
based on the Wiener-Hopf factorization.

For random walks in discrete time optimal stopping problems with
the reward functions  $G(x)=x^+$ and
$G(x)=({\rm e}^x-1)^+$ were solved already by Darling et al. in \cite{DLT}. 
In Mordecki \cite{M} the results in \cite{DLT} were lifted to the
framework of L\'evy processes with the focus to the reward functions
$G(x)=(K-{\rm e}^x)^+$ (put option) and $G(x)=({\rm e}^x-K)^+$ (call
option). Furthermore, for the value function $G(x)=(x^+)^\nu,\nu>0,$
the problem was solved by Novikov and Shiryayev in \cite{NS} and \cite{NS2} for random walks and L\'evy processes using
Appell polynomials and Appell functions,
see also Kyprianou and Surya \cite{KS} and  Alili and Kyprianou \cite{AK}. In addition to the Wiener-Hopf factorization, the fluctuation identities constitute an essential tool in these solutions.

This technique was generalized  under certain conditions to be applicable for more general
reward functions independently 
in Surya \cite{S} and Deligiannidis et al. \cite{DLU}. In particular, in \cite{S} it is assumed 
that there exists a function $Q$ such that the  reward function has the representation 
\begin{equation}
\label{GQ}
G(x)=\mathbb{E}_0\Big(Q(x+M_T)\Big),
\end{equation}
where $T$ is an exponentially with parameter $\beta>0$ distributed
random variable independent of $X$ and  $M_t:=\sup_{0\leq s\leq t}X_s$
is the running maximum of the L\'evy process $X$. The function $Q$ should
have a unique positive root $x^*$ such that $Q(x)<0$ for
$x<x^*$ and $Q(x)>0$ and non-decreasing for $x>x^*$.  Then the
stopping problem (\ref{OSP}) is, in fact, one-sided and
the optimal stopping time is
$$
\tau^*:=\inf\{t  \,:\, X_t\geq x^*\}.
$$
Moreover, the value function is expressed in \cite{S}
as
$$
V(x)=\E_0\left(Q(x+M_T) 1_{\{x+M_T\geq x^*\}}\right).
$$

In \cite{DLU} the
starting point is the assumption that there exists a function $g$ such
that 
\begin{equation}
\label{Gg}
G(x)=\mathbb{E}_0\Big(g(x+X_T)\Big).
\end{equation}
To see the connection between $Q$ in (\ref{GQ}) and $g$ in (\ref{Gg})
recall the Wiener-Hopf factorization
\begin{equation}
\label{WiHo}
X_T\stackrel{(d)}=M_T+I^\circ_T,
\end{equation}
where $I^\circ_T\stackrel{(d)}=I_T:=\inf_{0\leq t\leq T}X_t$ and
independent of $X.$ Consequently,
\begin{equation}
\label{QR}
Q(x)=\E_0\Big(g(x+I_T)\Big).
\end{equation}
Since in \cite{DLU} the stopping problem is analyzed via the 
function 
$$
x\rightarrow {\E}_{0}(g(x+I_T))
$$ 
the relationship (\ref{QR}) implies that the results in \cite{S} 
and \cite{DLU} are very close to each other.

Our main result characterizes the solution of the optimal stopping
problem \eqref{OSP} similarly as in \cite{NS}, \cite{S} and
\cite{DLU} but is valid for very general strong Markov processes,
e.g., for diffusions, L\'evy processes and continuous time
Markov chains. The main ingredient in our approach is the representation of $\beta$-excessive functions in the form
\[x\mapsto \mathbb{E}_x\Big(\sup_{0\leq t\leq T}{f}(X_t)\Big)\]
for a function ${f}\geq 0.$ This kind of representations have been studied in F\"ollmer
and Knispel
\cite{FK} and El Karoui and Meziou \cite{KM}.

The structure of the paper is the following: In the next section we
state and prove the main theorems 
which lead to explicit solutions to the optimal stopping problem
\eqref{OSP} 
for general Hunt processes $X$ and reward functions $G$ with some
representation properties. 
Depending on the form of the representing function $\hat{f}$ the optimal stopping time is
one-sided or two-sided. 
We furthermore present in section 2.2 a way for finding $\hat{f}$ explicitly. In section 2.3 the connection to the representing measure approach is discussed. 
In Section \ref{ex} the results are illustrated with a variety of examples. 
\newpage

\section{Main results}\label{main}
\subsection{Verification theorems}
We use the notation from the introduction and let  $X=(X_t)_{t\geq
  0}$ be a Hunt process taking values in a subset $S$ of
${\R}$, e.g., an interval or a finite set.
We refer to \cite{MS} for references and further discussions on Hunt
processes. For a full account see Blumenthal and Getoor \cite{BG},
Chung and Walsh \cite{CW} and Sharpe \cite{Sh}. In particular, $X$ has
the strong Markov property and is quasi left continuous with right
continuous sample paths with left limits. For simplicity, we
suppose that the 
lifetime of $X$ is a.s. infinite. 

\begin{remark}
In some problems it is of interest to study $X,$ e.g., up to the
passage time of a level. In such cases $X$ 
is viewed as a Hunt process with lifetime equal to the passage
time. For instance, 
a Brownian motion $X$ killed at the first passage time 
\[\zeta_0:=\inf\{t\geq 0: X_t<0\}\]
is a Hunt process with lifetime $\zeta_0$. It is not difficult to
see how to refine our approach to cover these and more general
lifetimes, and we leave this task to the interested reader. 
\end{remark}

Throughout the paper, the notation
$T$ is used for an exponentially distributed random variable assumed
to be independent of $X$ and having the mean $1/\beta$ where $\beta >0$
is the discounting rate introduced in (\ref{OSP}). We remark that a
function $u$  is $\beta$-excessive if it is $0$-excessive, i.e, excessive,
for $X$ killed at $T.$ 

The following lemma is the key to our approach for the optimal
stopping problem (\ref{OSP}). 
It is proved similarly as Proposition 2.1 in \cite{FK}. We give a short proof for the sake of completeness.

\begin{lemma}\label{lem:exc}
Let $f:S\mapsto \R_+$ be an upper semicontinuous function and define
$u(x):=\mathbb{E}_x\big(\sup_{0\leq t\leq T}f(X_t)\big).$ Then the function
$u: S\mapsto \mathbb{R}\cup\{+\infty\}$ is $\beta$-excessive.
\end{lemma}

\begin{proof}
By the upper semicontinuity of $f$ and the right continuity of the
sample paths of $X$ the function $t\mapsto \sup_{0\leq s\leq t}f(X_s)$ is right
continuous and, hence, measurable. It follows that $u$ is
well-defined. Consider for $t>0$ 
\begin{align*}
u(x)&=\E_x\Big(\sup_{0\leq s\leq T}{f}(X_s)\Big)\\
&\geq \mathbb{E}_x\Big(\sup_{t\leq s\leq T}{f}(X_s);t\leq
T\Big)=\mathbb{E}_x\Big(\mathbb{E}_{x}(\sup_{t\leq s\leq T}{f}(X_s)
|\mathcal{F}_t);t\leq T\Big)\\
&=\mathbb{E}_x\Big(\mathbb{E}_{x}\big(\sup_{t\leq s\leq T}{f}(X_s)|X_t\big);t\leq T\Big)=\mathbb{E}_x\Big(u(X_t);t\leq T\Big)\\
&=\mathbb{E}_x\Big({\rm e}^{-\beta t}u(X_t)\Big),
\end{align*}
and, therefore, 
\[\mathbb{E}_x\Big(u(X_t);t\leq T\Big)\leq u(x).\]
Next,
\begin{align*}
\lim_{t\rightarrow 0}\mathbb{E}_x\Big({\rm e}^{-\beta t}u(X_t)\Big)&=\lim_{t\rightarrow 0}\mathbb{E}_x\Big(\sup_{t\leq s\leq T}{f}(X_s);t\leq T\Big)\\
&=\mathbb{E}_x\Big(\sup_{0\leq s\leq T}{f}(X_s)\Big)\\
&=u(x),
\end{align*}
where the second equality holds by the monotone convergence.
\end{proof}
\begin{example}
\label{diffusion}
Consider a regular linear diffusion $X$ on $\mathbb{R}$. Let $x\rightarrow
u_{\beta}(x, y)$ be the $\beta$-resolvent kernel associated with $X.$ 
As is well-known, the function $x\rightarrow u_{\beta}(x, y)$  is $\beta$-excessive for any fixed $y$.
We find the representation $u_{\beta}(\cdot, y)$ as expected
supremum. Recall that 
\begin{equation}
\label{e1}
\mathbb{E}_{x}\big({\rm e}^{-\beta H_y}\big)=\frac{u_{\beta}(x,y)}{u_{\beta}(y,y)}=\left\{
\begin{array}{rl}
\frac{\psi_{\beta}(x)}{\psi_{\beta}(y)},&\text{if}\quad x{\leq} y,\\
\frac{\varphi_{\beta}(x)}{\varphi_{\beta}(y)},&\text{if}\quad x {\geq} y,
\end{array}\right\}
\end{equation}
where $H_y:=\inf\{t: X_t=y\}$ and $\psi_{\beta}$ and $\varphi_{\beta}$
are the fundamental increasing and decreasing, 
respectively, solutions for the generalized differential operator associated with $X$, see \cite{BS}.
Introducing 
\begin{equation*}
x\rightarrow f_y(x)=\left\{
\begin{array}{rl}
u_{\beta}(y, y),&\text{if }\quad x=y,\\
0,&\text{otherwise},
\end{array}\right\}
\end{equation*}
we have by (\ref{e1})
\begin{align*}
\mathbb{E}_{x}\Big(\sup_{t\leq T}f_{y}(X_t)\Big)&=u_{\beta}(y, y)\,\mathbb{P}_{x}(H_y < T)\\
&=u_{\beta}(y,y)\,\mathbb{E}_{x}({\rm e}^{-\beta H_y})\\
&=u_{\beta}(x, y).
\end{align*}
For another example, consider the $\beta$-excessive function
$\psi_{\beta}\wedge 1.$ We assume that there exists a point $y$
such that $\psi_{\beta}(y)=1.$ Notice that  $\psi_{\beta}(x)\geq 1$
for all $x\geq y$ since $\psi_{\beta}$ is non-decreasing. Define
\begin{equation*}
x\rightarrow f_y^{\uparrow}(x)={\bf 1}_{[y,+\infty)}(x).
\end{equation*}
Then, calculating as above, we have  
\begin{equation*}
\mathbb{E}_{x}\Big(\sup_{t\leq T}f_{y}^{\uparrow}(X_t)\Big)=\left\{
\begin{array}{rl}
 1,&\text{if}\quad x\geq y,\\
\displaystyle{\psi_{\beta}(x)},&\text{if}\quad x \leq y.
\end{array}\right\}
\end{equation*}
\end{example}
We consider now the optimal stopping problem \eqref{OSP} and assume throughout the paper that the reward function $G$ is non-negative, lower semicontinuous and satisfies
\begin{equation}\label{cond_g}
{\E}_{x}\Big(\sup_{t\geq 0}{\rm e}^{-\beta t}G(X_t)\Big)<\infty
\end{equation}
Our main result which characterizes solutions of
\eqref{OSP} is as follows:
\begin{thm} 
\label{thm:one-sided}
(One-sided case). Assume that there exists an upper semicontinuous function
$\hat{f}:S\mapsto\R$
 and a point $x^*\in S$ such that 
\begin{enumerate}[(a)]
	\item \label{ii}	
          \begin{enumerate}[(i)]
		\item $\hat{f}(x)\leq 0$ for $x\leq x^*$,
		\item $\hat{f}(x)$ is positive and  non-decreasing for $x>x^*$.
	\end{enumerate}
	\item \label{c}
		\begin{enumerate}[(i)]
		\item $\mathbb{E}_x\Big(\sup_{0\leq t\leq T}\hat{f}(X_t)\Big)=G(x)$ for $x\geq x^*$
		\item $\mathbb{E}_x\Big(\sup_{0\leq t\leq T}\hat{f}(X_t)\Big)\geq G(x)$ for $x\leq x^*$\label{eq:cii}.
	\end{enumerate}
\end{enumerate}
Then the value function of the optimal stopping problem (\ref{OSP}) is given by
\begin{equation}
\label{value_fct}
V(x)=\mathbb{E}_x\Big(\sup_{0\leq t\leq T}\hat{f}(X_t)1_{\{X_t\geq x^*\}}\Big)
=\E_x\left(\hat{f}(M_T)1_{\{M_T\geq x^*\}}\right)
\end{equation}
and
\[\tau^*:=\inf\{t\geq 0: X_t> x^*\}\]
is an optimal stopping time.
\end{thm}

\begin{proof}
Since the function $x\mapsto\hat{f}(x)1_{\{x\geq x^*\}}$ is
non-decreasing, the second equality in \eqref{value_fct} holds. Let
$\hat V$ denote the function given via these expressions, i.e.,
$$
\hat V(x):=\mathbb{E}_x\Big(\sup_{0\leq t\leq T}\hat{f}(X_t)1_{\{X_t\geq x^*\}}\Big).
$$
By Lemma \ref{lem:exc}, we know that $\hat V$ is $\beta$-excessive, 
so that $({\rm e}^{-\beta t}\hat V(X_t))_{t\geq 0}$ is a non-negative c\`{a}dl\`{a}g
supermartingale (see \cite{CW} p. 101). 
Hence by optional sampling we have for all $x$ and all stopping times $\tau$
\begin{equation}\label{eq:supermart}
\hat V(x)\geq \E_x\Big({\rm e}^{-\beta \tau}\hat V(X_\tau)\Big).
\end{equation}
By condition (\ref{ii}) we furthermore have $\hat{f}(x)1_{\{x\geq
  x^*\}}\geq \hat{f}(x)$ for all $x$, 
therefore using \eqref{c} we obtain from \eqref{eq:supermart}
\begin{equation*}
\hat V(x)\geq \E_x\Big({\rm e}^{-\beta \tau}G(X_\tau)\Big).
\end{equation*}
To conclude the proof we show that 
\begin{equation}\label{eq:leq}
\hat V(x)\leq \E_x\Big({\rm e}^{-\beta \tau^*}G(X_{\tau^*})\Big).
\end{equation}
The value function $V$ being the smallest $\beta$-excessive
majorant of $G$ it then follows that $\hat V=V$ and $\tau^*$ is an
optimal stopping time, as claimed.
To prove \eqref{eq:leq} consider
\begin{align*}
\hat V(x)&=\E_x\Big(\sup_{0\leq t\leq T}\hat{f}(X_t)1_{\{X_t\geq x^*\}}\Big)\\
&=\E_x\Big(\sup_{0\leq t\leq T}\hat{f}(X_t)1_{\{X_t\geq x^*\}};\tau^*<T\Big)\\
&=\E_x\Big(\hat{f}(M_T)1_{\{M_T\geq x^*\}};\tau^*<T\Big)\\
&=\E_x\Big(\hat{f}(M_T);\tau^*<T\Big)\\
&=\E_x\Big(\hat{f}(\sup_{\tau^*\leq t\leq T}X_t);\tau^*<T\Big)\\
&=\E_x\Big({\rm e}^{-\beta\tau^*}\E_{X_{\tau^*}}\big(\hat{f}(\sup_{0\leq t\leq T}X_t)\big)\Big)
\end{align*}
by the strong Markov property.
For the last equality note that by the memoryless property of the exponential distribution
\begin{align*}
&\E_x\Big(\hat{f}(\sup_{\tau^*\leq t\leq T}X_t);\tau^*<T\Big)\\
=&\int_{0}^{\infty} \E_x\Big(\E_x\Big(\hat{f}(\sup_{\tau^*\leq t\leq r}X_t)|\mathcal{F}_{\tau^*}\Big){1}_{\{\tau^*\leq r\}}\Big)\Pro(T\in dr)\\
=&\E_x\Big(\E_{X_{\tau^*}}\Big(\int_{0}^{\infty}\hat{f}(\sup_{0\leq t\leq r}X_t)\Pro(T\in dr)\Big){\rm e}^{-\beta\tau^*}\Big)\\
=&\E_x\Big({\rm e}^{-\beta\tau^*}\E_{X_{\tau^*}}\big(\hat{f}(\sup_{0\leq t\leq T}X_t)\big)\Big).
\end{align*}
Since $\hat{f}(\sup_{0\leq t\leq T}X_t)\leq \sup_{0\leq t\leq T}\hat{f}(X_t)$ we obtain
\begin{align*}
\hat V(x)&\leq \E_x\Big({\rm e}^{-\beta\tau^*}\E_{X_{\tau^*}}\big(\sup_{0\leq t\leq T}\hat{f}(X_t)\big)\Big)\\
&=\E_x\Big({\rm e}^{-\beta\tau^*}G(X_{\tau^*})\Big),
\end{align*}
where we used \eqref{c}(ii) in the last step. This shows \eqref{eq:leq} and the claim is proved.
\end{proof}

\begin{remark}\label{rem:left-sided}
By applying Theorem \ref{thm:one-sided} to the process $-X$, the
reward function $x\mapsto G(-x)$ and 
to the function $x\mapsto\hat{f}(-x)$ one obtains the analogous result
for a threshold to the left applicable, e.g., to put-type options.
\end{remark}

Although one-sided stopping rules as in the previous theorem often arise naturally in many examples (especially for monotonic reward functions and discounting, see also \cite{CIN} for a general discussion), for other problems a two-sided stopping rule is needed. The following theorem gives the corresponding result for this case.

\begin{thm} (Two-sided case)\label{thm:two-sided}  Assume that there exists a lower semicontinuous function
$\hat{f}:S\mapsto\R$ and points $x_*,\,x^* \in S$ such that 
\begin{enumerate}[(a)]
	\item 
	\begin{enumerate}[(i)]
		\item $\hat{f}(x)\leq 0$ for $x_*<x<x^*$,
		\item $\hat{f}(x)$ is non-increasing for $x<x_*$ and non-decreasing for $x>x^*$,
	\end{enumerate}
	\item 
		\begin{enumerate}[(i)]
		\item $\E_x\Big(\sup_{0\leq t\leq T}\hat{f}(X_t)\Big)=G(x)$ for $x\not\in[x_*, x^*]$,
		\item $\E_x\Big(\sup_{0\leq t\leq T}\hat{f}(X_t)\Big)\geq G(x)$ for $x\in[x_*, x^*]$.
	\end{enumerate}
\end{enumerate}
Then the value function $V$ of the optimal stopping problem (\ref{OSP}) satisfies
\begin{align}
\label{2-sid}
\nonumber
V(x)=&\E_x\Big(\sup_{0\leq t\leq
  T}\hat{f}(X_t)1_{\{X_t\not\in[x_*,x^*]\}}\Big)
\\
=&
\E_x\Big(\hat{f}(I_T)1_{\{I_T\leq x_*\}}\vee \hat{f}(M_T)1_{\{M_T\geq x^*\}}\Big)
\end{align}
and
\[\tau^*:=\inf\{t\geq 0: X_t\not\in[x_*,x^*]\}\]
is an optimal stopping time. 
\end{thm}

\begin{proof}
This theorem is proved similarly as Theorem \ref{thm:one-sided}. The second equality in
(\ref{2-sid}) follows from the monotonic properties of $\hat f.$ 
Letting 
$$
\hat V(x):=\E_x\Big(\sup_{0\leq t\leq
  T}\hat{f}(X_t)1_{\{X_t\not\in[x_*,x^*]\}}\Big)
$$
it is seen from Lemma \ref{lem:exc}  that $\hat V$ is
$\beta$-excessive and, by conditions (a) and (b), 
$$\hat V(x)\geq \mathbb{E}_x\Big({\rm e}^{-\beta\tau}G(X_{\tau})\Big)$$
for all stopping times $\tau\in\mathcal{M}.$ 
It remains to show that
\begin{equation}
\label{e12}
\hat V(x)\leq \mathbb{E}_x\Big({\rm e}^{-\beta\tau^*}G(X_{\tau^*})\Big).
\end{equation}
For this, consider
\begin{align*}
\hat V(x)&=\E_x\Big(\sup_{0\leq t\leq T}\hat{f}(X_t)1_{\{X_t\geq x^*\}} \vee \sup_{0\leq t\leq T}\hat{f}(X_t)1_{\{X_t\leq x_*\}};\tau^*<T\Big)\\
&=\E_x\Big(\hat{f}(M_T)1_{\{M_T\geq x^*\}}\vee\hat{f}(I_T)1_{\{I_T\leq x_*\}} ;\tau^*<T\Big)\\
&=\E_x\Big(\hat{f}(\sup_{\tau^*\leq t\leq T}X_t)\vee \hat{f}(\inf_{\tau^*\leq t\leq T}X_t) ;\tau^*<T\Big)\\
&=\E_x\Big({\rm e}^{-\beta\tau^*}\E_{X_{\tau^*}} \big(\hat{f}(\sup_{0\leq t\leq T} X_t)) \vee (\hat{f}(\inf_{0\leq t\leq T} X_t)\big)\Big),
\end{align*}
where the last equality is obtained by the strong Markov property.
Since 
$$
\hat{f}(\sup_{0\leq t\leq T} X_t) \vee \hat{f}(\inf_{0\leq t\leq
  T} X_t)\leq \sup_{0\leq t\leq T}\hat{f} (X_t),
$$
we have
\begin{align*}
\hat V(x)&\leq \E_x\Big({\rm e}^{-\beta\tau^*}\E_{X_{\tau^*}}\big(\sup_{0\leq t\leq T}\hat{f}(X_t)\big)\Big)\\
&=\E_x\Big({\rm e}^{-\beta\tau^*}G(X_{\tau^*})\Big).
\end{align*}
Thus the proof is complete.
\end{proof}
\begin{remark} 
The methodologies in \cite{DLU} and \cite{S} seem to be applicable only for one-sided problems (for L\'evy processes). It is remarkable that the present approach provides a clean characterization also in two-sided cases (for general Hunt processes).
\end{remark}

\subsection{Computing $\hat{f}$ }\label{sec:find_f}
To apply Theorem \ref{thm:one-sided} and \ref{thm:two-sided} for a given reward function $G$ the question arises how to find a function $\hat{f}$ fulfilling the conditions stated therein. For the related problem for underlying L\'evy processes, Surya \cite{S} gave a general way for sufficiently regular
bounded
functions $G$ using the Fourier transform, but this seems to be specific for L\'evy processes and is based on a certain assumption on the zeros of the Wiener-Hopf factors of $X$. We now discuss a more general way.

Let $G$ be a given reward function and assume that for all $x$ 
\begin{equation}\label{eq:cond}
\E_x\Big({\rm e}^{-\beta t}G(X_t)\Big)\rightarrow0\quad\mbox{as $t\rightarrow\infty$.}
\end{equation}
Note that, by dominated convergence, (\ref{eq:cond})  is fulfilled if in addition to (\ref{cond_g}) it also holds 
\begin{equation*}
\lim_{t\rightarrow\infty}{\rm e}^{-\beta t}G(X_t)=0\quad\mathbb{P}_x\text{-a.s.}
\end{equation*}
Furthermore, assume that the function $G$ is in the domain of the
extended infinitesimal generator of $X$ killed 
at $T$ (cf. \cite[VII.1]{RY}). Then, by definition, there exists a function
$\tilde{f}$ such 
that 
$$\int_{0}^{t}|\tilde{f}(X_s)|ds <\infty\quad\mathbb{P}_x\text{-a.s}$$ 
and the process
$$\left({\rm e}^{-\beta t}G(X_t)-G(X_0)+\int_0^t {\rm e}^{-\beta s}\tilde{f}(X_s)ds\right)_{t\geq 0}$$
is a $(\mathcal{F}_t,\mathbb{P}_x)$-martingale. Consequently,
\[\E_x\Big({\rm e}^{-\beta t}G(X_t)\Big)-G(x)+\E_x\left(\int_0^t{\rm e}^{-\beta s}\tilde{f}(X_s)ds\right)=0,\]
and as $t\rightarrow\infty$ we obtain using \eqref{eq:cond} 
\[G(x)=\lim_{t\rightarrow\infty}\E_x\left(\int_0^t {\rm e}^{-\beta s}\tilde{f}(X_s)ds\right). \]
Hence, if, e.g., 
\begin{equation}\label{eq:cond2}
\nonumber
\mathbb{E}_x\left(\int_{0}^{\infty} {\rm e}^{-\beta s}|\tilde{f}(X_s)|ds\right)<\infty,
\end{equation}
the Lebesgue dominated convergence theorem yields
\begin{equation}
\label{eq:cond_DLU}
G(x)=\frac{1}{\beta}\,\E_x\left(\tilde{f}(X_T)\right).
\end{equation}
Notice that for $G$ in the domain of the infinitesimal generator $A$
of $X$ we have $\tilde{f}=-(A-\beta)G$. 
In this case equation \eqref{eq:cond_DLU} reflects one of the fundamental properties of the resolvent operator.
\begin{remark}
For optimal stopping problems for L\'evy processes it was assumed in
\cite{DLU} that a representation of type \eqref{eq:cond_DLU} (see also
\eqref{Gg}) holds for $G$. Hence, the above considerations are useful
when applying the methods in \cite{DLU}.
\end{remark}

Imitating the approach for optimal stopping of L\'evy processes, we
next express $G$ as the expectation of a function of $M_T$
(cf. (\ref{GQ}) and (\ref{Gg})). Obviously, it holds 
\begin{align*}
\E_x\left(\tilde{f}(X_T)\right)=
\E_x\left(\E_x\left(\tilde{f}(X_T)|M_T\right)\right).
\end{align*}
We assume that the regular conditional probability distribution
\begin{equation}\label{conditional_dist1}
\mathbb{P}_{x}(X_T\in dy| M_T=z):=\mathbb{P}_{x}(X_T\in dy\,,\, M_T\in
dz)/\mathbb{P}_{x}(M_T\in dz)
\end{equation}
exists and has the needed regularity properties so that it holds
\begin{align}
\label{fQ}
\frac{1}{\beta}\E_x\left(\tilde{f}(X_T)\right)=\mathbb{E}_{x}\Big(Q(M_T;x)\Big),
\end{align}
where 
\begin{equation}\label{conditional_dist}
Q(z;x):=\frac{1}{\beta}\int_{-\infty}^{z}\tilde{f}(y)\mathbb{P}_{x}(X_T\in dy|M_T=z).
\end{equation}

These observations are now applied to characterize solutions for optimal
stopping problems for L\'evy processes and linear diffusions via Theorem \ref{thm:one-sided}. 
L\'evy processes are discussed first. The following lemma is well
known. We present a short proof for the sake of completeness.
\begin{lemma}
\label{levy-sup}
For a L\'evy process $X$ (which is not a subordinator or a compound Poisson
process) it holds that
\begin{equation}\label{conditional_dist2}
\mathbb{P}_{x}(X_T\in dy| M_T=z)=\mathbb{P}_{z}(I_T\in dy).
\end{equation}
In particular, the conditional distribution does not depend on $x$.
\end{lemma}
\begin{proof} The claim follows from the Wiener-Hopf
  factorization (\ref{WiHo}) and the spatial homogeneity of $X.$ Indeed, 
\begin{align*}
\mathbb{P}_{x}(X_T\in dy\,,\, M_T\in dz)&=\mathbb{P}_{0}(x+X_T\in
dy\,,\, x+M_T\in dz)\\
\
&=\mathbb{P}_0 (x+M_T+I^{\circ}_{T}\in dy\,,\, x+M_T\in dz)\\
&=\mathbb{P}_0 (z+I^{\circ}_{T}\in dy\,,\, x+M_T\in dz)\\
&=\mathbb{P}_z(I_{T}\in dy)\,\mathbb{P}_x(M_T\in dz).
\end{align*}
\end{proof}
\begin{proposition}
\label{levy-stop}
Consider the optimal stopping problem (\ref{OSP}) in case of 
a L\'evy process $X$ (as introduced in Lemma \ref{levy-sup}). Assume that there exists a function
$\tilde{f}$ such that (\ref{eq:cond2}) holds and 
define
\begin{equation}\label{f_hat_zero1}
Q(z):=\frac 1\beta\,\int_{-\infty}^{0}\tilde{f}(z+y)\mathbb{P}_{0}(I_T\in dy).
\end{equation}
If  $Q$ fulfils (a) in Theorem \ref{thm:one-sided} then
the value function of the optimal 
stopping problem (\ref{OSP}) is given as in (\ref{value_fct}) with $\hat{f}=Q\vee 0.$
\end{proposition}

\begin{proof}
Since $Q$ is assumed to fulfil (a) in Theorem \ref{thm:one-sided}
also $\hat f$ fulfils this condition. We need to verify that $\hat f$
satisfies (b) in Theorem \ref{thm:one-sided}. Definition \eqref{f_hat_zero1} of $Q$, Lemma \ref{levy-sup}, and
formula (\ref{fQ}) yield
\begin{equation}
\label{fQ2}
\nonumber
G(x)=\frac 1\beta\,\E_x\left(\tilde{f}(X_T)\right)=\mathbb{E}_{x}\Big(Q(M_T)\Big).
\end{equation}
Consequently, since $Q\leq \hat f$ we have for all $x$
\begin{equation}
\label{fQ3}
G(x)=\mathbb{E}_{x}\Big(Q(M_T)\Big)\leq \mathbb{E}_{x}\Big(\hat f(M_T)\Big)\leq
\mathbb{E}_{x}\Big(\sup_{0\leq t\leq T} \hat f(X_t)\Big).
\end{equation}
In case $x>x^*$ - since $Q$ satisfies (a) in Theorem \ref{thm:one-sided} - (\ref{fQ3}) holds with
equalities instead of inequalities, and this completes the proof. 
\end{proof}

Next we consider the optimal stopping problem (\ref{OSP}) for a
regular linear diffusion. For simplicity, it is assumed that $X$ is
conservative and takes values on the whole $\mathbb{R}.$ The following
result is extracted from  \cite[II.19]{BS}. We use the notation 
from Example \ref{diffusion}.
\begin{lemma}
\label{diffusion-sup}
For a conservative regular linear diffusion $X$ on $\mathbb{R}$  it holds that
\begin{equation}
\label{conditional_dist3}
\mathbb{P}_{x}(X_T\in dy| M_T=z)= \frac{\psi_\beta(y)\,
  m(dy)}{\int_{-\infty}^z \psi_\beta(u)m(du)},
\end{equation}
where $m$ is the speed measure. In particular, the conditional distribution does not depend on $x$.
\end{lemma}
\noindent The next result can now be proved analogously as Proposition
\ref{levy-stop}.

\begin{proposition}
\label{diffusion-stop}
Consider the optimal stopping problem (\ref{OSP}) for 
a diffusion $X$ as introduced above. Assume that there exists a function
$\tilde{f}$ such that (\ref{eq:cond2}) holds and 
define
\begin{equation}\label{f_hat_zero2}
Q(z):=\frac 1{\beta}\,\frac{\int_{-\infty}^{z}\tilde{f}(y) {\psi_\beta(y)\,
  m(dy)}}{\int_{-\infty}^z \psi_\beta(y)\,m(dy)}.
\end{equation}
If  $Q$ fulfils (a) in Theorem \ref{thm:one-sided} then
the value function of the optimal 
stopping problem (\ref{OSP}) is given as in (\ref{value_fct}) with $\hat{f}=Q\vee 0.$
\end{proposition}
To apply Proposition \ref{diffusion-stop} it is useful to have
condition on hand to guarantee that $Q$ fulfills (a) in 
Theorem \ref{thm:one-sided}. The following Lemma gives such a
condition in 
terms of the function $\tilde{f}$ which is known explicitly in many situations of interest.
\begin{lemma}
\label{lemma:diffusion-stop}
In the setting of Proposition \ref{diffusion-stop} assume that there
exists $\tilde{x}$ 
such that $\tilde{f}(x)\leq 0$ for $x\leq \tilde{x}$ and
$\tilde{f}(x)$ is non-decreasing for $x\geq \tilde{x}$.  
Then $Q\leq 0$ or there exists $x^*$ such that $Q$ fulfils (a) in Theorem \ref{thm:one-sided}.
\end{lemma}
\begin{proof}
For simplicity we assume that the measure $m$ is absolutely continuous, i.e. $m(dy)=m(y)dy$, and let
\[x^*=\inf\{z:\int_{-\infty}^{z}\tilde{f}(y) \psi_\beta(y)\,m(y)dy >0\}.\]
If $x^*=\infty$, then the claim trivially holds. Therefore, we assume
$x^*<\infty$. It remains to prove that $Q$ is non-decreasing for
$x>x^*$. 
Obviously $x^*>\tilde{x}$ and therefore by the assumption on $\tilde{f}$ we get
\[
\int_{-\infty}^z(\tilde{f}(z)-\tilde{f}(y)) \psi_\beta(y)\,m(y)dy\geq
0.
\]
Using this fact, by differentiating the function $Q$ we immediately see that $\frac{d}{dx}Q(x)\geq 0$.
\end{proof}

\begin{remark}
\label{path-decomp}
It is critical that the function $Q$ defined in (\ref{f_hat_zero1})
for L\'evy processes and in (\ref{f_hat_zero2}) for diffusions does not
depend on $x.$ This clearly follows from the fact that the
conditional distribution of $X_T$ given $M_T$ does not depend on the
initial state  of the process. Path decomposition theorems for strong
Markov processes imply such results in general, see Millar \cite{millar78}.
This observation yields a method for finding $\hat{f}$ for general underlying processes.
\end{remark}

\subsection{Connection to the representing measure approach}
Another approach to the solution of optimal stopping problems is the observation that finding the solution of such problems is equivalent to finding the Choquet-type integral representation of the value function in terms of the resolvent kernel $u_\beta(x,y)$.  The Radon measure $\sigma$ that appears in this representation characterizes the $\beta$-excessive value function $V$, and furthermore, the stopping region can be described as the support of $\sigma$.
 
This approach was first discussed in \cite{salminen85} for optimal stopping problems of diffusions. In \cite{MS} a verification theorem based on this approach applicable for general Hunt processes is presented. Moreover, in \cite{MS} the appearance of a function of the maximum in the solution of optimal stopping problems for L\'evy processes is explained via this approach and the Wiener-Hopf factorization.   

The critical point to utilize the approach is to find a candidate
solution for the measure $\sigma$. In \cite{Sa}, this was carried out
explicitly for spectrally positive L\'evy processes in the
Novikov-Shiryaev problem in terms of 
the Appell polynomials of $X_T$ using the Laplace transforms. In
Proposition \ref{prop:repr_measure}  we extend these results by pointing out a connection between this approach and our approach in this paper for a general Hunt process $X$ which does not have negative jumps, that is, 
\begin{equation}\label{eq:one-sided}
\Pro_x\Big(\inf_{t\geq 0}(X_t-X_{t-})\geq 0\Big)=1\ \ \forall x\in S.
\end{equation}

We work in the setting of Hunt processes satisfying the set of assumptions 
as given in \cite[Section 2]{MS}. In particular, it is assumed that there exists a dual resolvent with respect to a duality measure $m$ and that Hypothesis (B) from \cite{KW} holds.
Moreover, we assume that the Laplace transform of the first hitting time $H_x:=\inf\{t: X_t=x\}$ admits for all $x_0< x$ the representation.
\begin{equation}\label{hit}
\E_x\big({\rm e}^{-\beta{H_{x_0}}}\big)=\frac{u_{\beta}(x,x_0)}{u_{\beta}(x_0,x_0)}\ \ \forall x_0<x\in S,
\end{equation}
where $T_x=\inf\{t>0:X_t=x\}$.
This representation is valid for a large class of processes. A sufficient condition is that (I) and (II) below hold true, see \cite[Chapter V.3 and Exercise VI.4.18]{BG}.
\begin{itemize}
\item[{ (I)}]\label{I} The mapping $x\mapsto \int_S u_\beta(x,y) f(y)m(dy)$ is continuous, where $f$ is a bounded and measurable function with compact support.
\item[{ (II)}]\label{II} Every point $x\in S$ is regular for $\{x\}$, that is 
\begin{equation}\label{eq:regular}
\Pro_{x}(T_{x}=0)>0,
\end{equation}
where $T_x=\inf\{t>0:X_t=x\}$.
\end{itemize}
Assumption (I) is (4.1) in \cite[p. 284]{BG}, and assumption (II) 
implies that $X$ has a local time at every point of 
the state space. Condition (\ref{hit}) furthermore holds true whenever $X$ is a spectrally positive L\'{e}vy process, see \cite[Theorem 3.12, Corollary 8.9]{Kyp}.


Note that, in case  $X$  have positive jumps, overshoot occurs in the problem. 
Nonetheless it is possible to find an explicit representation of the
measure $\sigma.$ 

\begin{proposition}\label{prop:repr_measure}
Let $X$ be a Hunt process as described above with no negative jumps that fulfills condition (\ref{hit}). Assume that for the optimal stopping problem (\ref{OSP}) there exists a function $\tilde{f}$ such that (\ref{eq:cond2}) holds. Moreover, assume that the function 
\[Q(z):=\frac{1}{\beta}\int_{-\infty}^{z}\tilde{f}(y)\mathbb{P}(X_T\in dy|M_T=z)\]
 fulfills (a) in Theorem \ref{thm:one-sided}. Then the value function
 $V$ of (\ref{OSP}) has the representation
\begin{align*}
V(x)=\int_{[x^*,\infty)\cap S}u_{\beta}(x,y)\sigma(dy)\ \ \forall\ x\in S,
\end{align*}
where the measure $\sigma$ is given by
\begin{equation}
\label{measure1}
\sigma(dy)=\left\{
\begin{array}{rl}
 \tilde{f}(y)m(dy),&\text{for}\quad y>x^*,\\
0,&\text{for}\quad  y<x^*,
\end{array}\right\}
\end{equation}
and
\begin{equation}
\label{measure2}
\sigma(\{x^*\})=\frac{\E_{x^*}\Big(\tilde{f}(X_T)1_{\{X_T\leq x^*\}}\Big)}{\beta u_{\beta}(x^*,x^*)}.
\end{equation}
In the particular case that $X$ is a spectrally positive L\'evy process, then $\sigma(\{x^*\})=0$, and 
\[V(x)=E_x\big(\tilde{f}(X_T)1_{\{X_T> x^*\}}\big).\]
\end{proposition}

\begin{proof}
 We remark first that similarly as in the proof of Proposition
 \ref{levy-stop} it can be verified that the value function $V$  of
 the problem satisfies the equality
$$
V(x)=\E_x\left(Q(M_T)1_{\{M_T\geq x^*\}}\right),
$$
where $x^*$ is the point associated with $Q$ via (a) in Theorem \ref{thm:one-sided}.
Now we have 
\begin{align*}
V(x)&=\E_x\Big(Q(M_T)1_{\{M_T\geq x^*\}}\Big)=\frac{1}{\beta}\E_x\Big(\E_x\Big(\tilde{f}(X_T)|M_T\Big)1_{\{M_T\geq x^*\}}\Big)\\
&=\frac{1}{\beta}\E_x\Big(\E_x\Big(\tilde{f}(X_T)1_{\{M_T\geq x^*\}}|M_T\Big)\Big)=\frac{1}{\beta}\E_x\Big(\tilde{f}(X_T)1_{\{M_T\geq x^*\}}\Big)\\
&=\frac{1}{\beta}\E_x\Big(\tilde{f}(X_T)1_{\{X_T> x^*\}}\Big)+\frac{1}{\beta}\E_x\Big(\tilde{f}(X_T)1_{\{M_T\geq x^*,X_T\leq x^*\}}\Big).
\end{align*}
Since it is assumed that $X$ moves continuously down we rewrite the second summand
\begin{align*}
\E_x\Big(\tilde{f}(X_T)1_{\{M_T\geq x^*,X_T\leq x^*\}}\Big)&=\E_x\Big(\tilde{f}(X_T)1_{\{T_{x^*}<T,X_T\leq x^*\}}\Big),
\end{align*}
and, by the strong Markov property,
\begin{align*}
\E_x\left(\tilde{f}(X_T)1_{\{M_T\geq x^*, X_T\leq x^*\}}\right)&=\E_x\left(\E_x\left(\tilde{f}(X_T)1_{\{X_T\leq x^*\}}|\mathcal{F}_{T_{x^*}}\right)1_{\{T_{x^*}<T\}}\right)\\
&=\E_x\left({\rm e}^{-\beta{T_{x^*}}}\E_{X_{T_{x^*}}}\left(\tilde{f}(X_T)1_{\{X_T\leq x^*\}}\right)\right)\\
&=\E_{x^*}\left(\tilde{f}(X_T)1_{\{X_T\leq x^*\}}\right)\E_x\left({\rm e}^{-\beta{T_{x^*}}}\right).
\end{align*}
Letting $c=\E_{x^*}\big(\tilde{f}(X_T)1_{\{X_T\leq x^*\}}\big)$ and using assumption \eqref{hit} we obtain
\begin{align*}
V(x)&=\frac{1}{\beta}\E_x\Big(\tilde{f}(X_T)1_{\{X_T> x^*\}}\Big)+\frac{1}{\beta}\E_x\Big(\tilde{f}(X_T)1_{\{M_T\geq x^*,X_T\leq x^*\}}\Big)\\
&=\int_{(x^*,\infty)\cap S}\tilde{f}(y)u_{\beta}(x,y)m(dy)+\frac{c}{\beta u_{\beta}(x^*,x^*)}u_{\beta}(x,x^*)\\
&=\int_{[x^*,\infty)\cap S}u_{\beta}(x,y)\sigma(dy)
\end{align*}
with $\sigma$ as in (\ref{measure1}) and in (\ref{measure2}). It remains to prove that $\sigma$ is indeed non-negative. To this end, we use the fact that each $\beta$-excessive function $u$ locally integrable with respect to the duality measure $m$ can be decomposed uniquely in the form
\begin{equation}\label{eq:decomposition}
u(x)=\int u_{\beta}(x,y)\sigma_u (dy) + h_u(x),
\end{equation}
where $\sigma_u$ is a non-negative Radon measure and $h_u$ is a $\beta$-harmonic function, see \cite[Theorem 2 p. 505 and Proposition 13.1 p. 523]{KW}. Since the value function $V$ is $\beta$-excessive we may apply \eqref{eq:decomposition} to obtain
\[\int u_{\beta}(x,y)\sigma_V (dy) + h_V(x)=V(x)=\int_{[x^*,\infty)\cap S}u_{\beta}(x,y)\sigma(dy),\]
where $\sigma_V$ is the non-negative Radon measure associated with $V$. By the Hahn decomposition we can write $\sigma-\sigma_V=\sigma^+-\sigma^-$ for some non-negative measures $\sigma^+$ and $\sigma^-$. Therefore,
\[h_V(x)+\int u_{\beta}(x,y)\sigma^- (dy)=\int u_{\beta}(x,y)\sigma^+(dy).\]
By the uniqueness of the decomposition in \eqref{eq:decomposition} we obtain $h_V=0$ and $\sigma^+=\sigma^-$. Hence, $\sigma=\sigma_V$ is indeed a non-negative measure. This proves the first claim.\\
\indent
For spectrally positive L\'{e}vy processes it is well known that $-I_T$ is $Exp(\lambda)$-distributed, where $\lambda$ is the unique positive root of the equation
\[\mathbb{E}({\rm e}^{\lambda X_1})={\rm e}^{\beta}.\]
Therefore, by the definition of $Q$ and Proposition \ref{conditional_dist2} we see that $Q$ is continuous. In particular, $Q(x^*)=0$. Using the Wiener-Hopf factorization \eqref{WiHo} we obtain
\begin{align*}
\E_{x^*}\big(\tilde{f}(X_T)1_{\{X_T\leq x^*\}}\big)&=\E_{0}\big(\tilde{f}(x^*+M_T+I^{\circ}_{T})1_{\{x^*+M_T+I^{\circ}_{T}\leq x^*\}}\big)\\
&=\int_{0}^{\infty}\E_{0}\big(\tilde{f}(x^*+t+I^{\circ}_{T})1_{\{t+I^{\circ}_{T}\leq 0\}}\big)\Pro(M_T\in dt).
\end{align*}
Therefore, $\E_{x^*}\big(\tilde{f}(X_T)1_{\{X_T<x^*\}}\big)=0$, since
\begin{align*}
\E_{0}\big(\tilde{f}(x^*+t+I^{\circ}_{T})1_{\{t+I^{\circ}_{T}\leq 0\}}\big)&=\int_{-\infty}^{-t}\tilde{f}(x^*+t+s)\lambda {\rm e}^{\lambda s}ds\\
&= {\rm e}^{-\lambda t}\int_{-\infty}^{0}\tilde{f}(x^*+y)\lambda {\rm e}^{\lambda y}dy\\
&={\rm e}^{-\lambda t}Q(x^*)=0.
\end{align*}
Consequently, 
\[\sigma(\{x^*\})=\frac{\E_{x^*}\Big(\tilde{f}(X_T)1_{\{X_T\leq x^*\}}\Big)}{\beta u_{\beta}(x^*,x^*)}=0.\]
\end{proof}

\begin{remark}
By inspecting the previous proof one can see that assumption \eqref{hit} can be weakened by assuming that \eqref{hit} holds for $x_0=x^*$ only.
\end{remark}

\section{Examples} \label{ex}
In this section we demonstrate the applicability of the above theory
by solving some optimal stopping problems for diffusions, L\'evy
processes, and continuous time Markov chains.

\subsection{Ornstein-Uhlenbeck process}
Let $X$ be an Ornstein-Uhlenbeck diffusion with parameter $\gamma>0$, i.e. the differential operator associated to $X$ is 
\[Af=\frac{1}{2}\frac{d^2}{dx^2}f-\gamma x\frac{d}{dx}f.\]
We consider the optimal stopping problem (\ref{OSP}) for $X$ with the reward function
\[G_n(x)={(x^+)}^n,\;n\in \N.\]
For general L\'evy-driven Ornstein-Uhlenbeck processes it was proved in \cite{CIN} that the optimal stopping rule is one-sided, but no explicit way for finding the boundary explicitly in general was given there. Here we want characterize the optimal boundary for the Ornstein-Uhlenbeck diffusion using our approach. We have
\[(\beta-A)x^n=x^{n-2}\left((\beta+\gamma n)x^2-\frac{n(n-1)}{2}\right)=:\tilde{f}_n(x).\]
Using It\^o's lemma it is seen that $\tilde{f}_n$ indeed has the mean value property for $X_T$, i.e.,
\[x^n=\E_x\Big(\tilde{f}_n(X_T)\Big)\ \ \forall\ x\]
Recall that for Ornstein-Uhlenbeck processes the speed measure $m$ is given by
\[m(dx)=2{\rm e}^{-\gamma x^2}dx\]
and 
\[\psi_\beta(x)={\rm e}^{\gamma x^2/2}\D_{-\beta/\gamma}(-x\sqrt{2\gamma}),\]
where $\mathcal{D}_\nu$ denotes the parabolic cylinder function, see
\cite[A1.24,II.7]{BS}. 
Defining $Q_n=Q$ as in \eqref{f_hat_zero2} we obtain
\[x^n=\E_x\Big(Q_n(M_T)\Big).\]
Using arguments similar to those in Lemma \ref{lemma:diffusion-stop}
it is easy to see that $Q$ has a unique positive root $x^*$, which is
given by the positive solution of the equation
\begin{equation}\label{eq:OU}
\int_{-\infty}^x\tilde{f}(y){\rm e}^{-\gamma y^2/2}\D_{-\beta/\gamma}(-y\sqrt{2\gamma})dy=0.
\end{equation}
An approximative value of $x^*$ can be found using
numerical methods. Notice that $Q_n$ can also attain positive values
on the negative half-line. However, this problem can be dealt with by considering the function
\[
\hat{f}(x)=
\left\{
\begin{array}{rl}
0,  & x\leq 0,\\
  Q(x), &x>0,
\end{array}\right\}
   \]
which obviously satisfies the requirements of Theorem \ref{thm:one-sided}. Therefore, we obtain that the stopping time
\[\tau^*=\inf\{t:X_t\geq x^*\}\]
is optimal.

\subsection{American put in a L\'evy market}
In this subsection we review the well known example of a perpetual
American put option in a L\'evy driven market in the light of our
approach (cf. the solutions in \cite{M} and \cite{AK}, see also
\cite{CI} for further discussions). Although the solution is
well-known we want to demonstrate how to find the representing
function $\hat{f}$ systematically with the approach discussed in
Subsection \ref{sec:find_f}. Since the reward function 
$G(x)=K-{\rm e}^x, K>0,$ is non-increasing we use the modified
version of Theorem \ref{thm:one-sided}, see Remark
\ref{rem:left-sided}. The American call option can be handled similarly.

The generator $A$ of $X$ is given by
\begin{eqnarray}
\label{eq:generator}
&&
\nonumber
Af(x)=\frac{c^2}{2}\frac{d^2}{dx^2}f(x)+ b\frac{d}{dx}f(x)
\\
&&\hskip2cm 
+\int_{\mathbb{R}} \left(f(x+y)-f(x)-y\frac{d}{dx}f(x)1_{\{|y|<1\}}\right)\pi(dy),
\end{eqnarray}
where $(b,c,\pi)$ is the L\'evy triple of $X$ and $f$ is assumed to be
in the domain of $A.$ We wish to apply $\beta-A$ to $G(x)=K-{\rm e}^x.$ For this ansatz to work we
assume that 
\begin{equation}
\label{assu}
\E_0({\rm e}^{X_1})<{\rm e}^\beta. 
\end{equation}
At the end we will see that the
formula obtained is a solution for general $X$. We have  
\begin{align*}
(\beta-A)G(x)&=\beta(K-{\rm e}^x)+{\rm e}^x\left(\frac{c^2}{2}+b+\int_{\mathbb{R}}({\rm e}^y-1-y1_{|y|<1})\pi(dy)\right)\\
&=\beta K-(\beta-\psi(1)){\rm e}^x=:\tilde{f}(x),
\end{align*}
where $\psi(1)=\log \E_0({\rm e}^{X_1})<\beta$. To find $\hat{f}$ we
calculate, cf. Proposition \ref{levy-stop},
\begin{align*}
Q(x)&=\frac{1}{\beta}\,\E_0\left(\tilde{f}(x+M_T)\right)=K-\frac{\beta-\psi(1)}{\beta}\E_0({\rm e}^{M_T}){\rm e}^x.
\end{align*}
By noting that
\begin{equation}
\label{for11}
\frac{\beta-\psi(1)}{\beta}=\frac{1}{\E_0({\rm e}^{X_T})}
\end{equation}
and applying the Wiener-Hopf factorization (\ref{WiHo}) we obtain 
\[Q(x)=K-\E_0({\rm e}^{I_T})^{-1}{\rm e}^x.\]
Now we have a candidate for $\hat{f}$ by our ansatz, i.e., $\hat
f=Q{\vee}0,$ and 
formula (\ref{for11}) also makes sense for general L\'evy processes
without the integrability
assumption (\ref{assu}). 
Indeed, one immediately sees that for general $X$ we have
\[
K-{\rm e}^x=\E_x\Big(Q(I_T)\Big).
\]
We can apply Theorem \ref{thm:one-sided} (together with Remark
\ref{rem:left-sided}) to $\hat{f}$ and obtain 
that $x^*=\log(K\E_0({\rm e}^{I_T}))$ is the optimal stopping boundary and the value function is given by 
\[V(x)=\E_x\left(\hat{f}(I_T)1_{\{I_T<x^*\}}\right)=\E_x\Big(\Big(K-\frac{{\rm e}^{\,I_T}}{\E_0({\rm e}^{\,I_T})}\Big)^+\Big).\]

\subsection{Novikov-Shiryaev optimal stopping problem}
\label{subsec:appell}
In \cite{NS} Novikov and Shiryaev solved the optimal stopping problem
\eqref{OSP} with $G_n(x)=(x^+)^n,\, n=1,2,\dots,$ for random walks and in \cite{KS}
Kyprianou and Surya found the solution for L\'evy processes. Here we
discuss shortly how the solution can be found with the methodology
presented in this paper. For simplicity, we assume that $\beta>0.$

Let $X$ be a L\'evy process with the generator $A$ as given in
\eqref{eq:generator}. We fix the exponent $n\in \{1,2,\dots\}$ for the
reward function $G_n$ and
assume that the L\'evy measure $\pi$ satisfies 
\begin{equation}
\label{finite_exp}
\int_{(-\infty, -1)\cup(1, +\infty)}|y|^n\pi(dy)<\infty.
\end{equation}
This condition implies that $\E_x(|X_t|^k)<\infty$ for all $x,\, t\geq
0, $ and $k=1,2,\dots,n.$ Using the independence of increments, it is
easily seen that $\E_x(X_t^n)$ is a polynomial in $t.$ Consequently,
also the moments of $X_T$ up to $n$ exist. 
\begin{remark}
\label{cond1}
In \cite{KS} the solution is found under the  weaker condition that
(only) the integral over $(1,+\infty)$ in (\ref{finite_exp}) is
finite. Under this weaker condition  the moments of $M_T$ up to $n$
exist but the moments of $I_T$ may not exist. Working
under the stronger condition reveals the r\^ole of the Appell
polynomials of $X_T$ which we find interesting and important.
In fact, since the function $Q_n,$ see (\ref{QNS}), that is, the Appell polynomial associated
with $M_T,$ found in this way satisfies conditions (a) and (b) in
Theorem \ref{thm:one-sided} it is clear that our solution is indeed valid
under the weaker condition and, hence, coincides with the solution in
by Novikov and Shiryaev.
\end{remark}

We follow the receipt in Section \ref{sec:find_f} and operate with
$\beta -A$ on $x^n.$ This yields
\begin{align*}
\tilde{f}_n (x)&:=\left(\beta-A\right)(x^n)\\
&=\beta x^n-\left(\frac{c^2}{2}n(n-1)x^{n-2}+bnx^{n-1}\right)\\
&\hskip2cm
-\int_{{\mathbb{R}}}\left((x+y)^n-x^n-ynx^{n-1}1_{\{|y|\leq 1\}}\right)\pi(dy),
\end{align*}
where the integral term is well defined due to
(\ref{finite_exp}). Clearly, $\tilde f_n$ is an $n$th order polynomial
and satisfies the DE 
\begin{equation}\label{recursive}
\frac{d}{dx}\tilde{f}_n (x)=n\tilde{f}_{n-1}(x)
\end{equation}
with $f_0=1.$ We claim that 
\begin{equation}\label{mean}
\mathbb{E}_x\Big(\tilde{f}_n(X_T)\Big)=\mathbb{E}_0\Big(\tilde{f}_n(X_T+x)\Big)= \beta\,x^n,
\end{equation}
which, then, together with (\ref{recursive}) shows that $x\mapsto
\tilde{f}_n (x)/\beta$ is the $n$th order Appell polynomial associated
with $X_T,$ (for properties of  the Appell polynomials, see \cite{NS} and \cite{Sa}).
To prove (\ref{recursive}) we use standard computations, legitimate by (\ref{finite_exp}), with the It\^o formula for L\'evy processes,
see, e.g., \cite [p. 612]{JYC},  to obtain 
\begin{align*}
&{\rm e}^{-\beta t}\E_x\left(X_t^n\right) - x^n \\
&\hskip2cm
=\E_x\left(\int_0^t{\rm e}^{-\beta
      s}\left[\frac{c^2}{2}n(n-1)X_s^{n-2}+ bnX_s^{n-1} -\beta
      X_s^{n}\right]\, ds\right)\\
&\hskip2.5cm
+\E_x\left(\int_0^t{\rm e}^{-\beta s}\int_{{\mathbb{R}}}
\left[\left(X_s+y\right)^n-X_s^n - y\,n X_s^{n-1}1_{\{|y|\leq 1\}}\right]\pi(dy)\,ds\right).
\end{align*}
Multiplying both sides with $\beta,$ letting $t\to\infty$ and using
that the moments of $X_t$ are polynomials in $t$  we obtain
\begin{align*}
-\beta x^n &=\E_x\left(\int_0^\infty\beta{\rm e}^{-\beta
      s}\left[\frac{c^2}{2}n(n-1)X_s^{n-2}+ bnX_s^{n-1} -\beta
      X_s^{n}\right]\, ds\right)\\
&\hskip1cm
+\E_x\left(\int_0^\infty \beta {\rm e}^{-\beta s}\int_{{\mathbb{R}}}
\left[\left(X_s+y\right)^n-X_s^n - y\,n X_s^{n-1}1_{\{|y|\leq
    1\}}\right]\pi(dy)\,ds\right),
\end{align*}
which proves (\ref{mean}). Next we apply Proposition
\ref{levy-stop} to find the function $Q_n$ and then $\hat f.$ We write
formula (\ref{f_hat_zero1}) therein as 
\begin{equation}
\label{QNS}
Q_n(x)=\int_{-\infty}^0 Q^{(X_T)}_n(x+y)\mathbb{P}_{0}(I_{T}\in dy),
\end{equation}
where $Q^{(X_T)}_n(x):=\tilde{f}_n (x)/\beta.$ Recall from \cite{Sa}
the identity
\begin{equation}
\label{WHP}
Q^{(X_T)}_n(x+y)=\sum_{k=0}^n{n\choose k}
Q^{(M_T)}_k(x)Q^{(I_T)}_{n-k}(y)
\end{equation}
where $Q^{(M_T)}_i$ and $Q^{(I_T)}_i$ denote the $i$th Appell polynomials of
$M_T$ and $I_T,$ respectively. Identity (\ref{WHP}) follows easily
from the Wiener-Hopf factorization and the definition of the Appell
polynomials. Substituting (\ref{WHP}) into (\ref{QNS}) and using the normalizations 
$$
Q^{(I_T)}_0\equiv 1\quad {\rm and}\quad
\mathbb{E}_0\Big(Q^{(I_T)}_i(I_T)\Big)=0\ {\rm for\ all\ } i=1,2,...,n
$$
yields $Q_n=Q^{(M_T)}_n.$ Hence, our approach leads to the same
characterization of the optimal stopping time in terms of
$Q^{(M_T)}_n$ as in \cite{NS}. The crucial property to be able to
apply Theorem \ref{thm:one-sided}, i.e., that $Q^{(M_T)}_n$ has a unique
positive zero $x^*$ and is non-decreasing for $x>x^*$, is proved in \cite{NS} Lemma 5. Consequently, it
now easily follows that 
\begin{equation*}
\hat{f}(x):=\left\{
\begin{array}{rl}
Q^{(M_T)}_n(x),&\quad x> 0,\\
0,&\quad x\leq 0,
\end{array}\right\}
\end{equation*}
satisfies conditions (a) and (b) in Theorem \ref{thm:one-sided} and
the solution of the optimal stopping problem results. 

Note that Proposition \ref{prop:repr_measure} together
 with the fact that  that $x\mapsto
\tilde{f}_n (x)/\beta$ is the $n$th order Appell polynomial associated
with $X_T$ reproduces the formula of the representing measure
for spectrally positive L\'evy processes 
obtained in \cite[Corollary 8]{Sa}.


\subsection{Absorbing Brownian motion}

Let $X$  be a Brownian motion on $\mathbb{R}_+$  absorbed at $0,$
i.e., $X_t=0$ for $t\geq H_0.$ We solve the optimal stopping
problem (\ref{OSP}) with $G_n(x)=x^n,~n\geq 1$. The first step is to
find  $\tilde{f}_n$ such that  (\ref{eq:cond_DLU}) holds. Guided by 
Subsection \ref{sec:find_f} we set
\begin{equation}
\label{appell1}
\tilde{f}_n(x)=(\beta-\frac 12\frac{d^2}{dx^2})G_n(x)=x^{n-2}\left(\beta
x^2-\frac{n(n-1)}{2}\right)=:\beta P_n(x).
\end{equation}
Instead of checking (\ref{eq:cond_DLU}) via a direct calculation,
we recall that $P_n$ is the $n$th order Appell polynomial associated
with a standard Brownian motion $(W_t)_{t\geq 0}$ started from 0 and
evaluated at $T,$ see \cite[Example 3.9]{Sa}. Hence, for all
$x$
\begin{equation}
\label{appell2}
\nonumber
x^n=\E_0\Big(P_n(x+W_T)\Big)=\E_x\Big(P_n(W_T)\Big)=\frac{1}{\beta}\E_x\Big(\tilde{f}_n(W_T)\Big),
\end{equation}
and, by the strong Markov property,
\begin{align*}
\E_x\Big(\tilde{f}_n(W_T)\Big)&=\E_x\Big(\tilde{f}_n(W_T)1_{\{T< H_0\}}\Big)+\E_x\Big(\tilde{f}_n(W_T)1_{\{T> H_0\}}\Big)\\
&
=\E_x\Big(\tilde{f}_n(X_T)1_{\{X_T>0\}})\Big)+\E_x\Big(\E_x(\tilde{f}_n(W_T)\,|\,W_{H_0})1_{\{T> H_0\}}\}\Big)\\
&
=\E_x\Big(\tilde{f}_n(X_T)1_{\{X_T>0\}}\Big).
\end{align*}
Consequently,  we have 
$$
G_n(x)=\frac{1}{\beta}\E_x\Big(\tilde{f}_n(X_T)1_{\{X_T>0\}}\Big),
$$
which is \eqref{eq:cond_DLU} if $n\not=2.$ To cover also the case
$n=2$ we abandon \eqref{eq:cond_DLU} and instead try to find a
function $Q_n$ such that 
$$
\E_x\Big(\tilde{f}_n(X_T)1_{\{X_T>0\}}\Big)
=\E_x\Big(Q_n(M_T)\Big).
$$
For this, the conditional distribution of $X_T$ given $M_T$ is needed,
see (\ref{conditional_dist}). Since $X$
is not conservative, this distribution has an
atom at 0 and, hence, the formula in Lemma \ref{diffusion-sup} is not
valid. However, from \cite[II.19]{BS} we may deduce     
$$
\mathbb{P}_x(M_T\in dz)/dz=\frac{\partial}{\partial z}\left(\frac{\psi_\beta(x)}{\psi_\beta(z)}\right)
$$
and
$$
\mathbb{P}_x(M_T\in dz, X_T\in
dy)=2\frac{\psi_\beta(x)\psi_\beta(y)}{\psi_\beta^2(z)}dz\,dy.
$$
where
$
\psi_\beta(x)={\rm sh}(x\sqrt{2\beta}),\ x\geq 0.
$
see ibid. p. 121. Consequently, the conditional density of $X_T$ given $M_T$ is given by 
\[
f_{X_T}(y|M_T=z)=\frac{2\psi_\beta(y)}{\psi^{\,\prime}_\beta(z)}.
\]
and, therefore,
\begin{align*}
\hskip-1cm & Q_n(z):=\frac{1}{\beta}\int_0^z\tilde{f}_n(y)\frac{2\psi_\beta(y)}{\psi^{\,\prime}_\beta(z)}dy\\
&\hskip1cm
=\frac{2}{\beta\sqrt{2\beta}\,{\rm ch}(z\sqrt{2\beta})}
\int_0^zy^{n-2}\left(\beta y^2-\frac{n(n-1)}{2}\right){\rm sh}(y\sqrt{2\beta})dy.
\end{align*}
Since the integrand changes its sign from $-$ to $+$ one can see
(after some easy calculations) that $Q_n$ satisfies \eqref{ii} of
Theorem \ref{thm:one-sided}. Although Proposition \ref{diffusion-stop}
is not directly applicable it can be proved analogously as therein that  \eqref{c} of
Theorem \ref{thm:one-sided} holds for $\hat{f_n}:=Q_n\vee 0$. Therefore, the optimal stopping boundary $x^*$ is given as the unique positive solution to the equation
\begin{equation}
\label{fe1}
\int_0^zy^{n-2}\left(\beta y^2-\frac{n(n-1)}{2}\right){\rm
  sh}(y\sqrt{2\beta})dy=0.
\end{equation}
Integrating by parts, (\ref{fe1}) is seen to be equivalent with 
\begin{equation}
\label{fe2}
z\,{\rm ch}(z\sqrt{2\beta})-\frac{n}{\sqrt{2\beta}}\,{\rm sh}(z\sqrt{2\beta})=0.
\end{equation}
Notice that for $n=1$ this equation does not have a positive solution
and the optimal stopping rule is to stop immediately -- in other
words, reward $G(x)=x$ is $\beta$-excessive. 

It is also interesting to observe that (\ref{fe2}) is of the form 
$$
G_n(z)\frac{d\psi(z)}{dz}-\psi(z)\frac{dG_n(z)}{dz}=0,$$
and, hence, our approach yields in this case the same characterization
of the stopping point as given in Remark (ii) p. 97 in \cite{salminen85}.

\subsection{Markov chain}
To see the general applicability of the theory we treat  in this
subsection an easy example with an underlying finite-state Markov chain. Let $(X_t)_{t\geq 0}$ be a continuous-time Markov chain with state space $E=\{1,...,4\}$ and transition rates as given in the following figure.
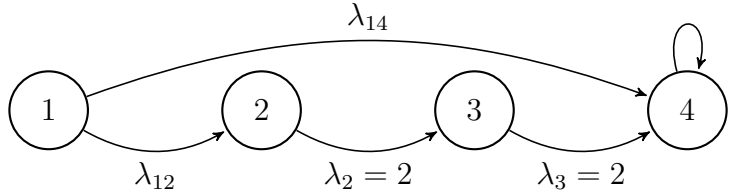
\begin{center}
\begin{tikzpicture}[->,>=stealth',shorten >=1pt,auto,node distance=2.8cm,semithick]
\tikzstyle{every state}=[fill=white,draw=black,thick,text=black,scale=1]
\node[state]         (A){$1$};
\node[state]         (B) [right of=A] {$2$};
\node[state]         (C) [right of=B] {$3$};
\node[state]         (D) [right of=C] {$4$};
\path (A) edge  [bend right] node[below] {$\lambda_{12}$} (B);
\path (D) edge  [loop above]  (D);
\path (B) edge  [bend right] node[below] {$\lambda_2=2$} (C);
\path (C) edge  [bend right] node[below] {$\lambda_3=2$} (D);
\path (A) edge  [bend right=-20] node[above] {$\lambda_{14}$} (D);
\end{tikzpicture}
\end{center}
To obtain explicit results let us assume that $\beta=1$ and $G(1)=3$,
$G(2)=1=G(3)$ and $G(4)=2$. We could use the method described in
Subsection \ref{sec:find_f} to find $\hat{f}$ such that
$G(x)=\E_x(\sup_{t\leq T}\hat{f}(X_t))$ for all $x\in E$, but it may be
instructive to find it directly: Since $4$ is an absorbing state, we take $\hat{f}(4)=G(4)=2$. To find $\hat{f}(3)$ we make the ansatz that $\hat{f}(3)\leq \hat{f}(4)$ and obtain
\[1=G(3)=\E_3\left(\sup_{t\leq T}\hat{f}(X_t)\right)=\hat{f}(4)
\mathbb{P}_3(S_3\leq T)+\hat{f}(3)\mathbb{P}_3(S_3> T),\]
where the sojourn time $S_3$ at state 3 is exponentially distributed
with parameter 2 
and independent of $T$. Solving for $\hat{f}(3)$ yields 
\[\hat{f}(3)=\frac{\lambda_3+\beta}{\beta}(1-2\frac{\lambda_3}{\lambda_3+\beta})=-1.\]
By a similar calculation we obtain that 
\[\hat{f}(2)=\frac{(\beta+\lambda_2)^2}{\beta^2+2\beta\lambda_2}\left(1-\hat{f}(4)\frac{\lambda_2^2}{(\beta+\lambda_2)^2}\right)=\frac{1}{5}.\]
Since $1$ is a global maximum point of $G$ taking 
\[\hat{f}(1)=G(1)=3\]
yields the right representation. We see that the assumptions of Theorem \ref{thm:two-sided} are fulfilled (with $x_*=2$, $x^*=4$). Therefore
\[\tau^*=\inf\{t\geq 0:X_t\not=3\}\]
is an optimal stopping time.

\noindent {\bf Acknowledgements.\ }
This paper was partly written during the first author's stay at \AA bo Akademi in the project \textit{Applied Markov processes -- fluid queues, optimal stopping and population dynamics, Project number: 127719  (Academy of Finland)}.  He would like to express his gratitude for the hospitality and support.

\newpage

\bibliographystyle{plain}

\end{document}